\documentclass[leqno]{llncs}

\usepackage{amsfonts}
\usepackage{amsmath}
\usepackage{amssymb}
\usepackage{comment}
\usepackage{expdlist}
\usepackage[dvips]{graphicx}
\usepackage{subfigure}

\newcommand{\set}[1]{\left \{ #1 \right \}}                     
\newcommand{\setst}[2]{\left \{ #1 \mid #2 \right \}}           

\newcommand{\abs}[1]{| #1 |}
\newcommand{\val}[1]{||#1 ||}
\newcommand{\supp}[1]{\mathop{\textrm{supp}}\nolimits (#1)}

\renewcommand{\phi}{\varphi}
\renewcommand{\tilde}{\widetilde}

\renewcommand{\epsilon}{\varepsilon}

\providecommand{\supp}[1]{\mathop{\textrm{supp}}\nolimits \left(#1\right)}

\newcommand{\reffig}[1]{Fig.~\ref{fig:#1}}           
\newcommand{\refcor}[1]{Corollary~\ref{cor:#1}}      
\newcommand{\refth}[1]{Theorem~\ref{th:#1}}          
\newcommand{\reflm}[1]{Lemma~\ref{lm:#1}}            
\newcommand{\refsec}[1]{Section~\ref{sec:#1}}        

\DeclareMathOperator{\iso}{iso}
\DeclareMathOperator{\odd}{odd}
\DeclareMathOperator{\cluster}{cluster}

\renewenvironment{proof}{\par\noindent%
{\bf Proof.\par\nopagebreak}}{\unskip\nobreak\enskip$\square$\par\bigskip}

\begin{document}

\title{
    Triangle-Free 2-Matchings Revisited \\
}

\institute
{
    Moscow State University
}

\author
{
    Maxim Babenko
    \thanks{
        Email: \texttt{max@adde.math.msu.su}.
        Supported by RFBR grant 09-01-00709-a.
    },
    Alexey Gusakov
    \thanks{
        Email: \texttt{agusakov@gmail.com}
    },
    Ilya Razenshteyn
    \thanks{
        Email: \texttt{ilyaraz@gmail.com}
    }
}

\maketitle

\begin{abstract}
    A \emph{2-matching} in an undirected graph $G = (VG, EG)$ is
    a function $x \colon EG \to \set{0,1,2}$ such that
    for each node $v \in VG$ the sum of values $x(e)$
    on all edges $e$ incident to $v$ does not exceed~2.
    The \emph{size} of $x$ is the sum $\sum_e x(e)$.
    If $\setst{e \in EG}{x(e) \ne 0}$ contains no triangles then $x$ is called \emph{triangle-free}.

    Cornu\'ejols and Pulleyblank devised a combinatorial $O(mn)$-algorithm
    that finds a triangle free 2-matching of maximum size
    (hereinafter $n := \abs{VG}$, $m := \abs{EG}$)
    and also established a min-max theorem.

    We claim that this approach is, in fact, superfluous
    by demonstrating how their results may be obtained
    directly from the Edmonds--Gallai decomposition.
    Applying the algorithm of Micali and Vazirani we are able to find a maximum triangle-free
    2-matching in $O(m\sqrt{n})$-time.
    Also we give a short self-contained algorithmic proof of the min-max theorem.

    Next, we consider the case of regular graphs.
    It is well-known that every regular graph admits a perfect 2-matching.
    One can easily strengthen this result and prove that every $d$-regular graph
    (for $d \geq 3$) contains a perfect triangle-free 2-matching.
    We give the following algorithms for finding a perfect triangle-free 2-matching
    in a $d$-regular graph:
    an $O(n)$-algorithm for $d = 3$,
    an $O(m + n^{3/2})$-algorithm for $d = 2k$ ($k \ge 2$), and
    an $O(n^2)$-algorithm for $d = 2k + 1$ ($k \ge 2$).
\end{abstract}

\section{Introduction}
\label{sec:intro}

\subsection{Basic Notation and Definitions}

We shall use some standard graph-theoretic notation throughout the
paper. For an undirected graph $G$ we denote its sets of nodes and
edges by $VG$ and $EG$, respectively. For a directed graph we speak
of arcs rather than edges and denote the arc set of $G$ by $AG$. A
similar notation is used for paths, trees, and etc. Unless stated
otherwise, we do not allow loops and parallel edges or arcs in
graphs. An undirected graph is called \emph{$d$-regular} (or just
\emph{regular} if the value of $d$ is unimportant) if all degrees of its nodes are equal to $d$.
A subgraph of $G$ induced by a subset $U \subseteq VG$ is denoted by $G[U]$.

\subsection{Triangle-Free 2-Matchings}

\begin{definition}
    Given an undirected graph $G$, a \emph{2-matching} in~$G$ is
    a function $x \colon EG \to \set{0,1,2}$ such that
    for each node $v \in VG$ the sum of values $x(e)$
    on all edges $e$ incident to $v$ does not exceed~2.
\end{definition}

A natural optimization problem is to find,
given a graph $G$, a \emph{maximum} 2-matching $x$ in~$G$,
that is, a 2-matching of maximum \emph{size} $\val{x} := \sum_e x(e)$.
When $\val{x} = \abs{VG}$ we call $x$ \emph{perfect}.

If $\setst{e}{x(e) = 1}$ partitions into a collection of node-disjoint circuits of odd length
then $x$ is called \emph{basic}.
Applying a straightforward reduction one can easily see
that for each 2-matching there exists a basic 2-matching of the same or larger size
(see \cite[Theorem~1.1]{CP-80}).
From now on we shall only consider basic 2-matchings~$x$.

One may think of a basic 2-matching~$x$ as a collection of node disjoint \emph{double edges}
(each contributing~2 to $\val{x}$) and \emph{odd length circuits}
(where each edge of the latter contributes~1 to $\val{x}$).
See~\reffig{examples}(a) for an example.

Computing the maximum size $\nu_2(G)$ of a 2-matching in $G$ reduces to finding a maximum
matching in an auxiliary bipartite graph obtained by splitting the nodes of $G$.
Therefore, the problem is solvable in $O(m \sqrt{n})$-time with the help
of Hopcroft--Karp's algorithm \cite{HK-73} (hereinafter $n := \abs{VG}$, $m := \abs{EG}$).
A simple min-max relation is known (see \cite[Th.~6.1.4]{sch-03} for an equivalent statement):
\begin{theorem}
\label{th:min_max_2_matching}
    $\nu_2(G) := \min_{U \subseteq VG} \left(\abs{VG} + \abs{U} - \iso(G - U)\right)$.
\end{theorem}
Here $\nu_2(G)$ is the maximum size of a 2-matching in $G$,
$G - U$ denotes the graph obtained from $G$ by removing nodes $U$ (i.e. $G[VG - U]$)
and $\iso(H)$ stands for the number of isolated nodes in $H$.
The reader may refer to \cite[Ch.~30]{sch-03} and \cite[Ch.~6]{LP-86} for a survey.

\medskip

Let $\supp{x}$ denote $\setst{e \in EG}{x(e) \ne 0}$.
The following refinement of 2-matchings was studied
by Cornu\'ejols and Pulleyblank \cite{CP-80} in connection with the Hamilton cycle problem:
\begin{definition}
    Call a 2-matching $x$ \emph{triangle-free} if $\supp{x}$ contains no triangle.
\end{definition}

They investigated the problem of finding a maximum size triangle-free 2-matching,
devised a combinatorial algorithm, and gave an $O(n^3)$ estimate for its running time.
Their algorithm initially starts with $x := 0$ and then performs a sequence of \emph{augmentation steps}
each aiming to increase $\val{x}$. Totally, there are $O(n)$ steps and
a more careful analysis easily shows that the step can be implemented to run in $O(m)$ time.
Hence, in fact the running time of their algorithm is $O(mn)$.

\begin{figure}[t!]
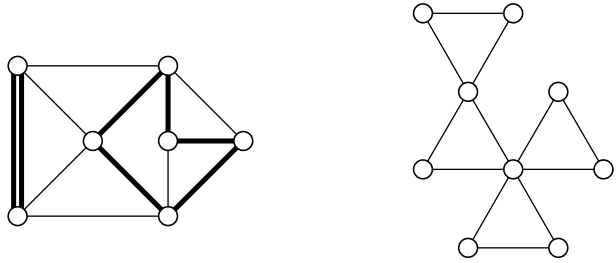

\label{fig:examples}
    \centering
    \includegraphics[trim=0 -12pt 0 0]{pics/example.1}%
    \hspace{2cm}
    \includegraphics{pics/cluster.1}%
    \caption{
        (a) A perfect basic 2-matching.
        (b) A triangle cluster.
    }
\end{figure}

The above algorithm also yields a min-max relation as a by-product.
Denote the maximum size of a triangle-free 2-matching in~$G$ by $\nu_2^3(G)$.
\begin{definition}
    A \emph{triangle cluster} is a connected graph whose edges partition into disjoint
    triangles such that any two triangles have at most one node in common
    and if such a node exists, it is an articulation point of the cluster.
    (See~\reffig{examples}(b) for an example.)
\end{definition}
Let $\cluster(H)$ be the number of the connected components of $H$ that are triangle clusters.
\begin{theorem}
\label{th:min_max_tf_2_matching}
    $\nu_2^3(G) := \min_{U \subseteq VG} \left(\abs{VG} + \abs{U} - \cluster(G - U)\right)$.
\end{theorem}
One may notice a close similarity between \refth{min_max_tf_2_matching} and
\refth{min_max_2_matching}.

\subsection{Our Contribution}

The goal of the present paper is to devise a faster algorithm
for constructing a maximum triangle-free 2-matching.
We give a number of results that improve the above-mentioned $O(mn)$ time bound.

Firstly, let $G$ be an arbitrary undirected graph.
We claim that the direct augmenting approach of Cornu\'ejols and Pulleyblank is, in fact, superfluous.
In \refsec{general} we show how one can compute a maximum triangle-free 2-matching with
the help of the Edmonds--Gallai decomposition \cite[Sec.~3.2]{LP-86}.
The resulting algorithm runs in $O(m\sqrt{n})$ time
(assuming that the maximum matching in $G$ is computed
by the algorithm of Micali and Vazirani \cite{MV-80}).
Also, this approach directly yields \refth{min_max_tf_2_matching}.

Secondly, there are some well-known results on matchings in regular graphs.
\begin{theorem}
\label{th:3_regular_perfect_matching}
    Every 3-regular bridgeless graph has a perfect matching.
\end{theorem}
\begin{theorem}
\label{th:bp_regular_perfect_matching}
    Every regular bipartite graph has a perfect matching.
\end{theorem}
The former theorem is usually credited to Petersen
while the second one is an easy consequence of Hall's condition.

\begin{theorem}[Cole, Ost, Schirra \cite{COS-01}]
\label{th:linear_bp_regular_perfect_matching}
    There exists a linear time algorithm that finds a perfect matching
    in a regular bipartite graph.
\end{theorem}
\refth{bp_regular_perfect_matching} and \refth{linear_bp_regular_perfect_matching}
imply the following:
\begin{corollary}
\label{cor:regular_perfect_2_matching}
    Every regular graph has a perfect 2-matching.
    The latter 2-matching can be found in linear time.
\end{corollary}

In \refsec{regular} we consider the analogues of \refcor{regular_perfect_2_matching}
with 2-matchings replaced by triangle-free 2-matchings.
We prove that every $d$-regular graph ($d \geq 3$) has a perfect triangle-free 2-matching.
This result gives a simple and natural strengthening to the non-algorithmic
part of \refcor{regular_perfect_2_matching}.

As for the complexity of finding a perfect 2-matching in a $d$-regular graph
it turns out heavily depending on $d$. The ultimate goal is
a linear time algorithm but we are only able to fulfill this task for $d = 3$.
The case of even $d$ ($d \ge 4)$ turns out reducible to $d = 4$, so
the problem is solvable in $O(m + n^{3/2})$ time by the use of the general algorithm
(since $m = O(n)$ for 4-regular graphs).
The case of odd $d$ ($d \ge 5$) is harder, we give an $O(n^2)$-time algorithm,
which improves the general time bound of $O(m \sqrt{n})$ when $m = \omega\left(n^{3/2}\right)$.

\section{General Graphs}
\label{sec:general}

\subsection{Factor-Critical Graphs, Matchings, and Decompositions}

We need several standard facts concerning maximum matchings (see~\cite[Ch.~3]{LP-86} for a survey).
For a graph $G$, let $\nu(G)$ denote the maximum size of a matching in $G$
and $\odd(H)$ be the number of connected components of $H$ with an odd number of vertices.

\begin{theorem}[Tutte--Berge]
\label{th:tutte_berge}
    $\nu(G) = \min_{U \subseteq VG} \frac12 \left(\abs{VG} + \abs{U} - \odd(G - U)\right)$.
\end{theorem}

\begin{definition}
    A graph $G$ is \emph{factor-critical} if for any $v \in VG$,
    $G - v$ admits a perfect matching.
\end{definition}

\begin{theorem}[Edmonds--Gallai]
\label{th:edmonds_gallai}
    Consider a graph $G$ and put
    $$
        \begin{array}{lll}
            D & := & \setst{v \in VG}{\mbox{there exists a maximum size matching missing $v$}},\\
            A & := & \setst{v \in VG}{\mbox{$v$ is a neighbor of $D$}},\\
            C & := & VG - (A \cup D).\\
        \end{array}
    $$
    Then $U := A$ achieves the minimum in the Tutte--Berge formula,
    and $D$ is the union of the odd connected components of $G[VG - A]$.
    Every connected component of $G[D]$ is factor-critical.
    Any maximum matching in $G$ induces a perfect matching in $G[C]$ and a matching in $G[VG - C]$
    that matches all nodes of $A$ to distinct connected components of $G[D]$.
\end{theorem}

We note that once a maximum matching~$M$ in $G$ is found, an Edmonds--Gallai decomposition
of $G$ can be constructed in linear time by running a search for an $M$-augmenting path.
Most algorithms that find~$M$ yield this decomposition as a by-product.
Also, the above augmenting path search may be adapted to produce
an \emph{odd ear decomposition} of every odd connected component of $G[VG - A]$:

\begin{definition}
    An \emph{ear decomposition} $G_0, G_1, \ldots, G_k = G$ of a graph $G$ is
    a sequence of graphs where $G_0$ consists of a single node,
    and for each $i = 0, \ldots, k - 1$, $G_{i+1}$ obtained from $G_i$ by adding
    the edges and the intermediate nodes of an ear.
    An \emph{ear} of $G_i$ is a path $P_i$ in $G_{i+1}$ such that
    the only nodes of $P_i$ belonging to $G_i$ are its (possibly coinciding) endpoints.
    An ear decomposition with all ears having an odd number of edges is called \emph{odd}.
\end{definition}

The next statement is widely-known and, in fact, comprises a part of the blossom-shrinking
approach to constructing a maximum matching.
\begin{lemma}
\label{lm:matching_all_but_one}
    Given an odd ear decomposition of a factor-critical graph $G$
    and a node~$v \in VG$ one can construct in linear time
    a matching $M$ in $G$ that misses exactly the node~$v$.
\end{lemma}

Finally, we classify factor-critical graphs depending
on the existence of a perfect triangle-free 2-matching.
The proof of the next lemma is implicit in \cite{CP-83} and one can easily
turn it into an algorithm:

\begin{lemma}
\label{lm:factor_critical_classification}
    Each factor-critical graph $G$ is either a triangle cluster
    or has a perfect triangle-free 2-matching $x$.
    Moreover, if an odd ear decomposition of $G$ is known
    then these cases can be distinguished and $x$ (if exists) can be constructed
    in linear time.
\end{lemma}

\subsection{The Algorithm}

For the sake of completeness, we first establish an upper
bound on the size of a triangle-free 2-matching.
\begin{lemma}
\label{lm:min_max_inequality}
    For each $U \subseteq VG$, $\nu_2^3(G) \leq \abs{VG} + \abs{U} - \cluster(G - U)$.
\end{lemma}
\begin{proof}
    Removing a single node from a graph $G$ may decrease $\nu_2^3(G)$ by
    at most~2. Hence, $\nu_2^3(G) \le \nu_2^3(G - U) + 2\abs{U}$.
    Also, $\nu_2^3(G - U) \le (\abs{VG} - \abs{U}) - \cluster(G - U)$
    since every connected component of $G - U$ that is a triangle cluster
    lacks a perfect triangle-free 2-matching.
    Combining these inequalities, one gets the desired result.
\end{proof}
The next theorem both gives an efficient algorithm a self-contained
proof of the min-max formula.
\begin{theorem}
\label{th:general_algo}
    A maximum triangle-tree 2-matching can be found in $O(m \sqrt{n})$ time.
\end{theorem}
\begin{proof}
    Construct an Edmonds--Gallai decomposition of $G$, call it $(D, A, C)$,
    and consider odd ear decompositions of the connected components of $G[D]$.
    As indicated earlier, the complexity of this step is dominated by
    finding a maximum matching~$M$ in $G$. The latter can be done in $O(m \sqrt{n})$ time (see~\cite{MV-80}).

    The matching $M$ induces a perfect matching $M_C$ in $G[C]$. We turn $M_C$ into double
    edges in the desired triangle-free 2-matching $x$ by putting $x(e) := 2$ for each $e \in M_C$.

    Next, we build a bipartite graph $H$. The nodes in the upper part of $H$ correspond
    to the components of $G[D]$, the nodes in the lower part of $H$ are just the nodes of~$A$.
    There is an edge between a component $C$ and a node $v$ in $H$ if and only if there is at least
    one edge between $C$ and $v$ in $G$.
    Let us call the components that are triangle clusters \emph{bad} and the others \emph{good}.
    Consider another bipartite graph $H'$ formed from $H$ by dropping all
    nodes (in the upper part) corresponding to good components.

    The algorithm finds a maximum matching $M_{H'}$ in $H'$ and then augments
    it to a maximum matching $M_H$ in $H$. This is done in $O(m \sqrt{n})$
    time using Hopcroft--Karp algorithm~\cite{HK-73}.
    It is well-known that an augmentation can only increase the set of matched
    nodes, hence every bad component matched by $M_{H'}$ is also matched by $M_H$ and vice versa.
    From the properties of Edmonds--Gallai decomposition it follows that $M_H$
    matches all nodes in $A$.

    Each edge $e \in M_H$ corresponds to an edge $\tilde e \in EG$,
    we put $x(\tilde e) := 2$.

    Finally, we deal with the components of $G[D]$.
    Let $C$ be a component that is matched (in~$M_H$) by, say, an edge $e_C \in M_H$.
    As earlier, let $\tilde e_C$ be the preimage of $e_C$ in $G$.
    Since $C$ is factor-critical, there exists a matching $M_C$ in $C$
    that misses exactly the node in $C$ covered by $\tilde e_C$.
    We find $M_C$ in linear time (see~\reflm{matching_all_but_one}) and put $x(e) := 2$ for each $e \in M_C$.

    As for the unmatched components, we consider good and bad ones separately.
    If an unmatched component $C$ is good, we apply \reflm{factor_critical_classification}
    to find (in linear time) and add to $x$ a perfect triangle-free 2-matching in $C$.
    If $C$ is bad, we employ \reflm{matching_all_but_one} and
    find (in linear time) a matching $M_C$ in $C$ that covers all the nodes
    expect for an arbitrary chosen one and set $x(e) := 2$ for each $e \in M_C$.

    The running time of the above procedure is dominated by constructing
    the Edmonds--Gallai decomposition of $G$ and finding matchings $M_{H'}$ and $M_H$.
    Clearly, it is $O(m \sqrt{n})$.

    It remains to prove that $x$ is a maximum triangle-free 2-matching.
    Let $n_{\rm bad}$ be the number of bad components in $G[D]$.
    Among these components, let $k_{\rm bad}$ be matched by $M_{H'}$ (and, hence, by $M_H$).
    Then $\val{x} = \abs{VG} - (n_{\rm bad} - k_{\rm bad})$.
    From K\"onig--Egervary theorem (see, e.g., \cite{LP-86}) there exists a vertex cover $L$ in $H'$
    of cardinality $k_{\rm bad}$ (i.e. a subset $L \subseteq VH'$ such that
    each edge in $H'$ is incident to at least one node in $L$).
    Put $L = L_A \cup L_D$, where $L_A$ are the nodes of $L$ belonging to the lower part of $H$
    and $L_D$ are the nodes from the upper part.
    The graph $G - L_A$ contains at least $n_{\rm bad} - \abs{L_D}$ components that
    are triangle clusters.
    (They correspond to the uncovered nodes in the upper part of~$H'$.
    Indeed, these components are only connected to $L_A$ in the lower part.)
    Hence, putting $U := L_A$ in \reflm{min_max_inequality} one gets
    $\nu_2^3(G) \le \abs{VG} + \abs{L_A} - (n_{\rm bad} - \abs{L_D}) =
    \abs{VG} + \abs{L} - n_{\rm bad} =
    \abs{VG} - (n_{\rm bad} - k_{\rm bad}) = \val{x}$.
    Therefore, $x$ is a maximum triangle-free 2-matching, as claimed.
\end{proof}

\section{Regular Graphs}
\label{sec:regular}

\subsection{Existence of a Perfect Triangle-Free 2-Matching}

\begin{theorem}
\label{th:existence_in_almost_regular}
    Let $G$ be a graph with $n - q$ nodes of degree $d$ and $q$ nodes of degree $d - 1$ ($d \ge 3$).
    Then, there exists a triangle-free 2-matching in $G$ of size at least $n - q/d$.
\end{theorem}
\begin{proof}
    Consider an arbitrary subset $U \subseteq VG$.
    Put $t := \cluster(G - U)$ and let $C_1, \ldots, C_t$ be the triangle cluster components of $G - U$.
    Fix an arbitrary component $H := C_i$ and let $k$ be the number of triangles in $H$.
    One has $\abs{VH} = 2k + 1$.
    Each node of $H$ is incident to either $d$ or $d - 1$ edges.
    Let $q_i$ denote the number of nodes of degree $d - 1$ in $H$.
    Since $\abs{EH} = 3k$ it follows that $(2k + 1)d - 6k - q_i = d + (2d - 6) k  - q_i \ge d - q_i$
    edges of $G$ connect $H$ to~$U$. Totally, the nodes in~$U$
    have at least $\sum_{i = 1}^t (d - q_i) \geq td - q$ incident edges.
    On the other hand, each node of $U$ has the degree of at most $d$, hence $td - q \le \abs{U} d$
    therefore $t - \abs{U} \le q/d$.
    By the min-max formula (see \refth{min_max_tf_2_matching}) this implies the
    desired bound.
\end{proof}

\begin{corollary}
\label{cor:existence_in_regular}
    Every $d$-regular graph ($d \ge 3$) has a perfect triangle-free 2-matching.
\end{corollary}

\subsection{Cubic graphs}

For $d = 3$ we speed up the general algorithm ultimately as follows:
\begin{theorem}
\label{th:regular_3_algo}
    A a perfect triangle-free 2-matching in a 3-regular graph can be found in linear time.
\end{theorem}
\begin{proof}
    Consider a 3-regular graph $G$.
    First, we find an arbitrary inclusion-wise maximal collection
    of node-disjoint triangles $\Delta_1, \ldots, \Delta_k$ in $G$.
    This is done in linear time by performing a local search at each node $v \in VG$.
    Next, we contract $\Delta_1, \ldots, \Delta_k$ into \emph{composite} nodes $z_1, \ldots, z_k$
    and obtain another 3-regular graph $G'$
    (note that $G'$ may contain multiple parallel edges).

    Construct a bipartite graph $H'$ from $G'$ as follows.
    Every node $v \in VG'$ is split into a pair of nodes $v^1$ and $v^2$.
    Every edge $\set{u, v} \in EG'$ generates edges $\set{u^1, v^2}$ and $\set{v^1, u^2}$ in $H'$.
    There is a natural surjective many-to-one correspondence between perfect matchings in $H'$
    and perfect 2-matchings in $G'$.
    Applying the algorithm of Cole, Ost and Schirra \cite{COS-01} to $H'$ we construct
    a perfect 2-matching $x'$ in $G'$ in linear time.
    As usual, we assume that $x'$ is basic, in particular $x'$ contains no circuit of length~2
    (i.e. $\supp{x'}$ contains no pair of parallel edges).

    Our final goal is to expand $x'$ into a perfect triangle-free 2-matching $x$ in $G$.
    The latter is done as follows. Consider an arbitrary composite node $z_i$
    obtained by contracting $\Delta_i$ in $G$.
    Suppose that a double edge $e$ of $x'$ is incident to $z_i$ in~$G'$.
    We keep the preimage of $e$ as a double edge of $x$ and add another double edge
    connecting the remaining pair of nodes in $\Delta_i$. See~\reffig{3regular}(a).

    Next, suppose that $x'$ contains an odd-length circuit $C'$ passing through $z_i$.
    Then, we expand $z_i$ to $\Delta_i$ and insert an additional pair of edges to $C'$.
    Note that the length of the resulting circuit $C$ is odd and is no less than~5.
    See~\reffig{3regular}(b).

    Clearly, the resulting 2-matching $x$ is perfect. But why is it triangle-free?
    For sake of contradiction, suppose that $\Delta$ is a triangle in $\supp{x}$.
    Then, $\Delta$ is an odd circuit in~$x'$ and no node of $\Delta$ is composite.
    Hence, $\Delta$ is a triangle disjoint from $\Delta_1, \ldots, \Delta_k$~--- a contradiction.
\end{proof}

\begin{figure}[t]
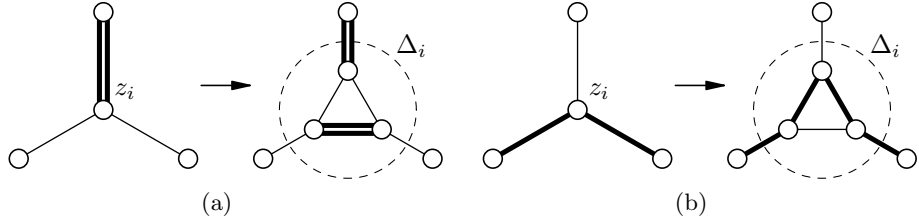

\label{fig:3regular}
    \centering
    \subfigure[]{
        \includegraphics{pics/3regular.1}%
    }
    \hspace{0.2cm}
    \subfigure[]{
        \includegraphics{pics/3regular.2}%
    }
    \caption{
        Uncontraction of $z_i$.
    }
\end{figure}

Combining the above connection between triangle-free 2-matchings in $G$
and 2-matchings in $G'$ with the result of Voorhoeve \cite{V-79} one can prove the following:
\begin{theorem}
    There exists a constant $c > 1$ such that every 3-regular graph $G$ contains at least $c^n$ perfect triangle-free 2-matchings.
\end{theorem}

\subsection{Even-degree graphs}

To find a perfect triangle-free 2-matching in a $2k$-regular graph $G$
($k \ge 2$) we replace it by a 4-regular spanning subgraph and then apply the general algorithm.

\begin{lemma}
\label{lm:2regular_subgraph}
    For each $2k$-regular ($k \ge 1$) graph $G$ there exists and can be found
    in linear time a 2-regular spanning subgraph.
\end{lemma}
\begin{proof}
    Since the degrees of all nodes in $G$ are even, $EG$
    decomposes into a collection of edge-disjoint circuits. This decomposition takes linear time.
    For each circuit $C$ from the above decomposition we choose
    an arbitrary direction and traverse $C$ in this direction turning
    undirected edges into directed arcs.
    Let $\overrightarrow G$ denote the resulting digraph. For each node $v$ exactly
    $k$ arcs of $\overrightarrow G$ enter $v$ and exactly $k$ arcs leave $v$.

    Next, we construct a bipartite graph $H$ from $\overrightarrow G$ as follows:
    each node $v \in \overrightarrow G$ generates a pair of nodes $v^1, v^2 \in VH$,
    each arc $(u,v) \in A\overrightarrow G$ generates an edge $\set{u^1,v^2} \in EH$.
    The graph $H$ is $k$-regular and, hence, contains a perfect matching~$M$
    (which, by \refth{linear_bp_regular_perfect_matching}, can be found in linear time).
    Each edge of $M$ corresponds to an arc of $\overrightarrow G$ and, therefore, to
    an edge of $G$. Clearly, the set of the latter edges forms a 2-regular
    spanning subgraph of $G$.
\end{proof}

\begin{theorem}
    A perfect triangle-free 2-matching in a $d$-regular graph ($d = 2k$, $k \ge 2$)
    can be found in $O(m + n^{3/2})$ time.
\end{theorem}
\begin{proof}
    Consider an undirected $2k$-regular graph $G$.
    Apply \reflm{2regular_subgraph} and construct find a 2-regular spanning subgraph $H_1$ of $G$.
    Next, discard the edges of $H_1$ and reapply \reflm{2regular_subgraph} thus
    obtaining another 2-regular spanning subgraph $H_2$ (here we use that $k \geq 2$).
    Their union $H := (VG, EH_1 \cup EH_2)$
    is a 4-regular spanning subgraph of~$G$.
    By \refcor{existence_in_regular} graph $H$ still contains a perfect triangle-free 2-matching $x$,
    which can be found by the algorithm from \refth{general_algo}.
    It takes $O(m)$ time to construct $H$ and $O(n^{3/2})$ time, totally
    $O(m + n^{3/2})$ time, as claimed.
\end{proof}

\subsection{Odd-degree graphs}

The case $d = 2k+1$ ($k \ge 2$) is more involved. We extract a spanning subgraph
$H$ of~$G$ whose node degrees are 3 and 4. A careful choice of $H$
allows us to ensure that the number of nodes of degree~3 is $O(n/d)$.
Then, by \refth{existence_in_almost_regular} subgraph~$H$ contains a nearly-perfect triangle-free 2-matching.
The latter is found and then augmented to a perfect one with the help
of the algorithm from \cite{CP-80}. More details follow.

\begin{lemma}
\label{lm:nearly_4_regular_subgraph}
    There exists and can be found in linear time a spanning subgraph~$H$
    of graph~$G$ with nodes degrees equal to 3 and 4.
    Moreover, at most $O(n/d)$ nodes in~$H$ are of degree~3.
\end{lemma}
\begin{proof}
    Let us partition the nodes of $G$ into pairs (in an arbitrary way) and
    add $n/2$ \emph{virtual} edges connecting these pairs. The resulting
    graph $G'$ is $2k+2$-regular. (Note that $G'$ may contain multiple
    parallel edges.)

    Our task is find a 4-regular spanning subgraph $H'$ of $G'$ containing at most $O(n/d)$ virtual edges.
    Once this subgraph is found, the auxiliary edges are dropped creating $O(n/d)$ nodes of degree~3
    (recall that each node of $G$ is incident to exactly one virtual edge).

    Subgraph~$H'$ is constructed by repeatedly pruning graph $G'$.
    During this process graph $G'$ remains $d'$-regular for some even $d'$
    (initially $d' := d + 1$).

    At each pruning step we first examine $d'$. Two cases are possible.
    Suppose $d'$ is divisible by~4, then a \emph{large} step is executed.
    The graph~$G'$ is decomposed into a collection of edge-disjoint circuits.
    In each circuit, every second edge is marked as \emph{red} while others are marked
    as \emph{blue}.
    This red-blue coloring partitions $G'$ into a pair of spanning $d'/2$-regular subgraphs.
    We replace $G'$ by the one containing the smallest number of virtual edges.
    The second case (which leads to a \emph{small} step) applies if $d'$ is not divisible by $4$.
    Then, with the help of \reflm{2regular_subgraph} a 2-regular spanning subgraph is found in $G'$.
    The edges of this subgraph are removed from $G'$, so $d'$ decreases by~2.

    The process stops when $d'$ reaches~4 yielding the desired subgraph~$H'$.
    Totally, there are $O(\log d)$ large (and hence also small) steps
    each taking time proportional to the number of remaining edges.
    The latter decreases exponentially, hence the total time to construct $H'$ is linear.

    It remains to bound the number of virtual edges in $H'$. There are exactly
    $t := \lfloor \log_2 (d+1)/4 \rfloor$ large steps performed by the algorithm.
    Each of the latter decreases the number of virtual edges in the current subgraph
    by at least a factor of~2. Hence, at the end there are $O(n/2^t) = O(n/d)$ virtual edges
    in $H'$, as required.
\end{proof}

\begin{theorem}
    A perfect triangle-free 2-matching in a $d$-regular graph ($d = 2k + 1$, $k \ge 2$)
    can be found in $O(n^2)$ time.
\end{theorem}
\begin{proof}
    We apply \reflm{nearly_4_regular_subgraph} and construct a subgraph~$H$ in $O(m)$ time.
    Next, a maximum triangle-free 2-matching $x$ is found in $H$,
    which takes $O(\abs{EH} \cdot \abs{VH}^{1/2}) = O(n^{3/2})$ time.
    By \refth{existence_in_almost_regular} the latter 2-matching obeys $n - \val{x} = O(n/d)$.
    To turn $x$ into a perfect triangle-free 2-matching in $G$ we apply the
    algorithm from \cite{CP-80} and perform $O(n/d)$ augmentation steps.
    Each step takes $O(m)$ time, so totally the desired perfect triangle-free 2-matching
    is constructed in $O(m + n^{3/2} + mn/d) = O(n^2)$ time.
\end{proof}

\subsection*{Acknowledgements}

The authors are thankful to Andrew Stankevich for fruitful suggestions
and helpful discussions.

\nocite{*}
\bibliographystyle{alpha}
\bibliography{main}

\end{document}